\documentclass[12pt]{article}

\usepackage{amsmath,amscd,amssymb,amsthm,tikz,enumerate,graphicx,verbatim}
\usepackage[top=3cm, bottom=3cm, left=2.54cm, right=2.54cm]{geometry}

\newif\ifcolor
\colortrue

\newcommand{\LG}{\mathop{\mathrm{LG}}}

\newcommand{\norm}[1]{\left \Vert #1 \right \Vert}
\newcommand{\st}{\,:\,}
\newcommand{\C}{\mathbb{C}}
\newcommand{\R}{\mathbb{R}}
\newcommand{\T}{\mathbb{T}}
\newcommand{\Z}{\mathbb{Z}}

\newcommand{\abs}[1]{\left\lvert#1\right\rvert}
\def\ls[#1,#2]{\overline{\vphantom{\vbox to 1.2 ex{}} #1\, #2}}
\newcommand\plist[1]{\bigl[ #1 \bigr]}

\newcommand{\vecx}{{\boldsymbol{x}}}
\newcommand{\vecy}{{\boldsymbol{y}}}
\newcommand{\vecu}{{\boldsymbol{u}}}
\newcommand{\vecv}{{\boldsymbol{v}}}
\newcommand{\vecw}{{\boldsymbol{w}}}
\newcommand{\vecz}{{\boldsymbol{z}}}
\renewcommand{\Re}{\mathop{\mathrm{Re}}}

\DeclareMathOperator{\var}{var}
\DeclareMathOperator{\cvar}{\rm cvar}
\DeclareMathOperator{\vf}{vf}

 \newtheorem{theorem}{Theorem}[section]
 \newtheorem{corollary}[theorem]{Corollary}
 \newtheorem{lemma}[theorem]{Lemma}

 \theoremstyle{definition}
 \newtheorem{example}[theorem]{Example}

 \newtheorem{definition}[theorem]{Definition}

\numberwithin{equation}{section}

\newcommand\email[1]{\texttt{#1}}

\begin{document}


\title{Isomorphisms of $AC(\sigma)$ spaces for linear graphs}

\date{}

\author{Shaymaa Al-shakarchi \and Ian Doust\footnote{
School of Mathematics and Statistics, University of New South Wales, UNSW Sydney 2052, Australia,
          \email{Email: i.doust@unsw.edu.au}
          }}

\maketitle

\begin{abstract}
  We show that among compact subsets of the plane which are drawings of linear graphs, two sets $\sigma$ and $\tau$ are homeomorphic if and only if the corresponding spaces of absolutely continuous functions (in the sense of Ashton and Doust) are isomorphic as Banach algebras. This gives an analogue for this class of sets to the well-known result of Gelfand and Kolmogorov for $C(\Omega)$ spaces.
  \hfill\break \smallskip
  
          \hfill\break
\end{abstract}

\section{Introduction}

The Banach algebras $AC(\sigma)$ and $BV(\sigma)$ were introduced in \cite{AD1} to provide a functional calculus model for operators on Banach spaces which is similar to, but more general than that for normal operators on a Hilbert space. Here $\sigma$ is a nonempty compact subset of the plane. The space $BV(\sigma)$ consists of those complex valued functions on $\sigma$ which are of bounded variation in the sense of \cite{AD1}, and $AC(\sigma)$ consists of the subalgebra of absolutely continuous functions. The class of $AC(\sigma)$ operators on a Banach space $X$ comprises those bounded operators which admit a bounded $AC(\sigma)$ functional calculus. This class includes all scalar-type spectral operators, and in the Hilbert space case, all normal operators (see \cite{AD2}). It is easy to check that if $T$ is an $AC(\sigma)$ operator then $T$ has spectrum $\sigma(T) \subseteq \sigma$.

If $\sigma$ is an interval $[a,b] \subseteq \R$ or the unit circle $\T$, the $AC(\sigma)$ operators are respectively the classes of well-bounded and trigonometrically well-bounded operators. At least on reflexive Banach spaces, there are well-developed structure theories for these classes which generalize the theories of self-adjoint and unitary operators (see \cite{Dow,BG}). Compact $AC(\sigma)$ operators admit a spectral representation similar to that for compact normal operators (see \cite{AD3}).

For more general compact sets $\sigma \subseteq \C$, a hurdle to advancing the theory of these operators has been our understanding of the structure of the corresponding function spaces.
The $AC(\sigma)$ spaces play a similar role in the theory of $AC(\sigma)$ operators to that played by the $C(\sigma)$ spaces in the theory of normal operators. It would be desirable therefore to have analogues of the powerful structure theorems that one has for these latter spaces.
%
The classical theorems of Banach--Stone and Gelfand--Kolmogorov state that if $\Omega_1$ and $\Omega_2$ are compact Hausdorff topological spaces then the following three conditions are equivalent.
\begin{enumerate}
 \item $\Omega_1$ is homeomorphic to $\Omega_2$.
 \item $C(\Omega_1)$ and $C(\Omega_2)$ are linearly isometric as Banach spaces.
 \item $C(\Omega_1)$ and $C(\Omega_2)$ are isomorphic as algebras (or as $C^*$-algebras).
\end{enumerate}
The situation for $AC(\sigma)$ spaces is more complicated, even in the restricted setting of the plane. To fix some terminology, we shall say that Banach algebras $\mathcal{A}$ and $\mathcal{B}$ are isomorphic (as Banach algebras), and write $\mathcal{A} \simeq \mathcal{B}$, if there exists a linear and multiplicative bijection $\Phi: \mathcal{A} \to \mathcal{B}$ such that $\Phi$ and $\Phi^{-1}$ are continuous. If $\Phi$ can be chosen  to be isometric, then we shall write $\mathcal{A} \cong \mathcal{B}$.

It was shown in \cite{DL1} that if $AC(\sigma_1)$ is isomorphic to $AC(\sigma_2)$ as algebras, then the isomorphism must necessarily be bi-continuous, and so the spaces are isomorphic as Banach algebras. Furthermore, if this is the case then $\sigma_1$ must be homeomorphic to $\sigma_2$. Unfortunately the converse to this result is false. Doust and Leinert \cite{DL1} showed that the $AC(\sigma)$ spaces for the closed unit disk and a closed square in $\C$ are not isomorphic. In \cite{DAS} it was shown that one can construct an infinite set of compact $AC(\sigma)$ operators $\{T_i\}_{i=1}^\infty$ on $\ell^1$ so that the sets $\sigma_i = \sigma(T_i)$ are all homeomorphic, but such that the algebras $AC(\sigma_i)$ (which are isomorphic to the operator algebras $\{f(T_i) \st f \in AC(\sigma_i)\}$) are mututally non-isomorphic. (A simpler example again will be given in Section~\ref{fin-ex}.)

On the other hand, if one suitably restricts the class of compact sets $\sigma$, positive results can be obtained.

\begin{theorem}[{\cite[Theorem 6.3]{DL1}}]\label{DL-theorem}
Suppose that $\sigma$ and $\tau$ are polygonal regions in $\C$ (that is, regions with polygonal boundary which are homeomorphic to the closed disk), then $AC(\sigma)$ is isomorphic to $AC(\tau)$ as a Banach algebra.
\end{theorem}

A more general result holds for polygonal regions with finitely many polygonal `holes'.

Another natural class of sets to examine are those which are in some sense one dimensional. It is straightforward to adapt the proof of
Theorem~\ref{DL-theorem}  to show that if $\sigma$ and $\tau$ are simple closed polygonal curves (or polygonal chains), then   $AC(\sigma)$ is isomorphic to $AC(\tau)$ as a Banach algebra.  (Indeed it follows from \cite{AD4} and Theorem~\ref{LG-norm} below that these spaces are all isomorphic to $AC(\T)$.)

It is straightforward to construct examples of operators whose spectrum is composed of one dimensional pieces but which are topologically more complicated than a line segment or a loop. For example, if $H$ denotes the discrete Hilbert transform on $\ell^p(\Z)$ (with $1 < p < \infty$) then it follows from \cite{DS} that the operator $T\in B(\ell^p(\Z) \oplus\ell^p(\Z))$ defined by $T(x,y) = (Hx, iHy)$ has a cross-shaped spectrum $\sigma(T) = \sigma = [-\pi,\pi] \cup i[-\pi,\pi]$. Furthermore, $T$ is an $AC(\sigma)$ operator which is not a scalar-type spectral operator unless $p  =2$.


The aim of this paper is to prove a Gelfand--Kolmogorov type theorem which covers all compact subsets which are made up of line segments.

\begin{definition} The class $\LG$ of `linear graphs' consists of those compact connected subsets of the plane which are the union of a finite number of (compact) line segments.
\end{definition}

Our main result, which is proved in Section~\ref{AC-isom}, is the following.

\begin{theorem}\label{main-theorem} Suppose that $\sigma,\tau \in \LG$. Then $AC(\sigma) \simeq AC(\tau)$ if and only if $\sigma$ is homeomorphic to $\tau$.
\end{theorem}

To prove Theorem~\ref{main-theorem} one needs to show that if $\sigma,\tau \in \LG$ are homeomorphic then there exists at least one homeomorphism $\phi: \sigma \to \tau$ such that $f \in AC(\sigma)$ if and only if $\Phi(f) = f \circ \phi^{-1} \in AC(\tau)$. In the proof of Theorem~\ref{DL-theorem} the corresponding map $\phi$ could be chosen to be the restriction of a homeomorphism of the whole plane. This homeomorphism of the plane was constructed as a composition of a finite number of simple homeomorphisms called locally piecewise affine maps and the main step was to show that each of these maps generated an isomorphism between the appropriate spaces of absolutely continuous functions. It is not too hard to see that such a strategy is not possible in the current setting since it may be that none of the homeomorphisms between $\LG$ sets $\sigma$ and $\tau$ are the restrictions of homeomorphisms of the plane.
To get around this problem we shall use some concepts from graph theory which provide suitable tools for dealing with homeomorphic pairs of sets in $\LG$.


By definition, any $\sigma \in \LG$ can be written as a finite union of line segments, $\sigma = \bigcup_{j=1}^m s_j$. We shall always impose the condition that any two distinct line segments $s_i$ and $s_j$ can only intersect at the endpoints of the line segments. Such a representation will be called proper. (It is a simple matter to replace a decomposition of $\sigma$ which does not satisfy this condition with one which does.) The set $\sigma$ can then be thought of as providing a drawing of a graph whose vertex set consists of the endpoints of these line segments. It might be noted, and indeed it will be useful to us below, that any such set $\sigma$ will admit many proper representations and hence many different graphs.

In Section~\ref{BV-AC} we briefly recall the technical definitions of the spaces $BV(\sigma)$ and $AC(\sigma)$. Section~\ref{LG} includes the relevant background from the topological theory of planar graphs which will be needed. In the final section we shall give an example which indicates that Theorem~\ref{main-theorem} can not be extended to cover all sets which are homeomorphic to a set in $\LG$.

Throughout the paper we shall frequently need to consider the restriction of a function $f: \sigma_1 \to \C$ to a subset $\sigma \subseteq \sigma_1$. Unless there is a risk of confusion we shall not notationally distinguish the function from its restriction. Where appropriate we shall use $\norm{f}_{\infty,\sigma}$ to denote $\sup_{\vecz \in \sigma} | f(\vecz)|$. We shall denote the number of elements of a set $A$ by $|A|$.

\section{$BV(\sigma)$ and $AC(\sigma)$}\label{BV-AC}

Throughout this section $\sigma$ will denote a nonempty compact subset of the plane, which we shall identify as $\C$ or $\R^2$ as is convenient.
We shall denote the closed line segment joining two points $\vecx,\vecy \in \C$ by $\ls[x,y]$.

For the remainder of this section assume that $f$ is a complex-valued function on $\sigma$.
Our first task is to define the variation (in the sense of \cite{AD1}) of $f$ on $\sigma$.

Let $S = \plist{\vecx_0,\vecx_1,\dots,\vecx_n}$ be a finite ordered list of elements of $\sigma$, where, for the moment, we shall assume that $n \ge 1$. Let $\gamma_S$ denote the piecewise linear curve joining the points of $S$ in order.
Note that the elements of such a list do not need to be distinct (although one may assume that consecutive elements of the list are different).

The \textit{curve variation of $f$ on the list $S$} is defined to be
\begin{equation*}
    \cvar(f, S) =  \sum_{i=1}^{n} \abs{f(\vecx_{i}) - f(\vecx_{i-1})}.
\end{equation*}
Associated to the list $S$ is a variation factor $\vf(S)$ which is roughly the greatest number of times that $\gamma_S$ crosses any line in the plane.
To make this more precise we need the concept of a crossing segment.

\begin{definition}\label{crossing-definition}
Suppose that $\ell$ is a line in the plane. For $0 \le i < n$ we say that $\ls[\vecx_i,\vecx_{i+1}]$ is a \textit{crossing segment} of $S$ on $\ell$ if any one of the following holds:
\begin{enumerate}[(i)]
  \item $\vecx_i$ and $\vecx_{i+1}$ lie on (strictly) opposite sides of $\ell$.
  \item $i=0$ and $\vecx_i \in \ell$.
  \item $i > 0$, $\vecx_i \in \ell$ and $\vecx_{i-1} \not\in \ell$.
\end{enumerate}
Let $\vf(S,\ell)$ denote the number of crossing segments of $S$ on $\ell$.
\end{definition}

\begin{definition}\label{vf-definition}
The \textit{variation factor} of $S$ is defined to be
  \[  \vf(S) = \max \{ \vf(S,\ell) \st \text{$\ell$ is a line in $\C$}\}. \]
\end{definition}

Clearly $1 \le \vf(S) \le n$. For completeness, in the case that
$S = \plist{\vecx_0}$ we set $\cvar(f, \plist{\vecx_0}) = 0$ and let $\vf(\plist{\vecx_0}) = 1$.

\begin{definition}\label{2d-var}
\begin{enumerate}[(i)]
\item The \textit{two-dimensional variation} of $f$ is defined to be
\[
    \var(f, \sigma) = \sup_{S}
        \frac{ \cvar(f, S)}{\vf(S)},
\]
where the supremum is taken over all finite ordered lists $S$  of elements of $\sigma$.

\item The \textit{variation norm} of $f$ is
  $ \norm{f}_{BV(\sigma)} = \norm{f}_\infty + \var(f,\sigma)$.
\end{enumerate}
\end{definition}

The  set of functions of bounded variation on $\sigma$ is then
  \[ BV(\sigma) = \{ f: \sigma \to \C \st \norm{f}_{BV(\sigma)} < \infty\}. \]
This space is always a Banach algebra under pointwise operations \cite[Theorem 3.8]{AD1}. If $\sigma = [a,b] \subseteq \R$ then the above definition is equivalent to the more classical one (see \cite[Proposition~3.6]{AD1}).

Let $\mathbb{P}_2$ denote the space of complex polynomials in two real variables of the form $p(x,y) = \sum_{n,m} c_{nm} x^n y^m$, and let $\mathbb{P}_2(\sigma)$ denote the restrictions of elements of $\mathbb{P}_2$ to $\sigma$ (now considered as a subset of $\R^2$). By \cite[Corollary 3.14]{AD1},  $\mathbb{P}_2(\sigma)$ is always a subalgebra of $BV(\sigma)$.

\begin{definition}
The set of \textit{absolutely continuous} functions on $\sigma$, denoted $AC(\sigma)$, is the closure of $\mathbb{P}_2(\sigma)$ in $BV(\sigma)$.
\end{definition}

The set $AC(\sigma)$ forms a closed subalgebra of $BV(\sigma)$ and hence is also a Banach algebra. Again, if $\sigma = [a,b]$ it turns our that this definition agrees with the classical definition. In general, $AC(\sigma)$ contains all sufficiently nice functions. It is shown in \cite{DL2} that
 $C^1(\sigma) \subseteq AC(\sigma) \subseteq C(\sigma)$, where one interprets $C^1(\sigma)$ as consisting of all functions for which there is a $C^1$ extension to an open neighbourhood of $\sigma$. Further details can be found in \cite{AD1,DL2}.

It is easy to see that these definitions are stable under invertible affine transformations of the plane. That is, if  $h(z) = \alpha z + \beta$ (where $\alpha,\beta \in \C$, $\alpha \ne 0$), then $BV(h(\sigma)) \cong BV(\sigma)$ and $AC(h(\sigma)) \cong AC(\sigma)$.

\section{The $LG$ norm}\label{LG}

In general, the nature of the $BV(\sigma)$ norm makes it challenging to compute. In this section we shall show that for sets $\sigma \in \LG$ one can work with a simpler quantity which, while not submultiplicative, is equivalent to $\norm{\cdot}_{BV(\sigma)}$ as a Banach space norm.

Suppose that $\sigma \in \LG $ has a proper representation $\sigma = \bigcup_{j=1}^m s_j$. 
For each $j$, $BV(s_j) \cong BV[0,1]$ and $AC(s_j) \cong AC[0,1]$. In particular, one can calculate $\var(f,s_j)$ using essentially the classical definition of variation rather than using Definition~\ref{2d-var}.

\begin{definition} For $f : \sigma \to \C$, let
   $\displaystyle \norm{f}_{LG(\sigma)} = \norm{f}_\infty + \sum_{j=1}^m \var(f,s_j)$.
\end{definition}

At first glance it might appear that this definition depends on the choice of representation of $\sigma$ as a union of line segments. As we noted above, this representation is not unique. It will be a consequence of the results in the next section however that the value of $\norm{f}_{LG(\sigma)}$ is independent of the representation.

\begin{theorem}\label{LG-norm}
$\norm{\cdot}_{LG(\sigma)}$ is a norm and is equivalent to $\norm{\cdot}_{BV(\sigma)}$.
\end{theorem}

\begin{proof} To check that $\norm{\cdot}_{LG(\sigma)}$ is a norm is straightforward.
Note that since $\var(f,s_j) \le \var(f,\sigma)$ for each $j$, we clearly have that
  \[ \norm{f}_{LG(\sigma)} \le \norm{f}_\infty + m \var(f,\sigma) \le m \norm{f}_{BV(\sigma)}. \]
The reverse inequality is more difficult.

We proceed by induction on $m$. If $m = 1$ then the two norms are identical, so there is nothing to prove. Suppose then that there exists $C_m$ such that given any $\tau = \bigcup_{j=1}^m s_j \in \LG$ which is the union of $m$ line segments then
$\norm{f}_{BV(\tau)} \le C_m \norm{f}_{LG(\tau)}$. Suppose  that $\tau$ is such a set, and that $s = \ls[\vecx,\vecy]$ is a line segment such that $\tau \cap s$ is either $\{\vecx\}$ or else $\{\vecx,\vecy\}$. Let $\tau_1 = \tau \cup s$.

Suppose that $f: \tau_1 \to \C$ and that $S = [\vecz_0,\vecz_1,\dots,\vecz_n]$ is a list
of points in $\tau_1$. Our aim is to bound $\cvar(f,S)/\vf(S)$.

For $j = 1,\dots,n$, let $\ell_j = \ls[\vecz_{j-1},\vecz_j]$. Define subsets $I_1,I_2,I_3 \subseteq \{1,2,\dots,n\}$ by (see Figure~\ref{ex1})
  \begin{align*}
    I_1 &= \{ j \st \vecz_j,\vecz_{j-1} \in \tau\}, \\
    I_2 &= \{ j \st \vecz_j,\vecz_{j-1} \in s\}, \\
    I_3 &= \{1,2,\dots,n\} \setminus (I_1 \cup I_2).
  \end{align*}
Note that if $j \in I_3$ then neither endpoint of $\ell_j$ is equal to either $\vecx$ or $\vecy$.

\begin{figure}[ht]
\begin{center}

%
\ifcolor
\begin{tikzpicture}[scale=0.75]

 \draw[thick,black] (0,0) -- (2,-2) -- (5,1) -- (3.5,3) -- (0,3) -- (0,0) -- (-2,0.5) -- (-1.5,-2) -- (0,0);
 \draw[thick,black] (3.5,3) -- (2,2);
 \draw[ultra thick,red]  (0,0) -- (2,2);
 
\draw[black] (0,0) node[circle, draw, fill=black!50,inner sep=0pt, minimum width=4pt] {};
 \draw[black] (0,-0.1) node[below] {$\vecx$};
 \draw[black] (2,2) node[circle, draw, fill=black!50,inner sep=0pt, minimum width=4pt] {};
 \draw[black] (1.9,2) node[above] {$\vecy$};

 \draw[black] (-2,-1) node {$\tau$};
 \draw[red] (2.3,1.75) node[below] {$s = \ls[\vecx,\vecy]$};

\end{tikzpicture}
\hspace{2cm}
\begin{tikzpicture}[scale=0.75]

 \draw[thick,black] (0,0) -- (2,-2) -- (5,1) -- (3.5,3) -- (0,3) -- (0,0) -- (-2,0.5) -- (-1.5,-2) -- (0,0);
 \draw[thick,black] (3.5,3) -- (2,2);
 \draw[ultra thick,red ]  (0,0) -- (2,2);

 \draw[thick,dashed,blue] (0.5,-0.5) -- (0.5,0.5) -- (1,1) -- (0,2.5) -- (1,3) -- (4.25,2) -- (2,2) -- (1.4,1.4) -- (3.5,-0.5) -- (0,1.3) -- (-1,0.25);

 \draw[black] (0.5,-0.5) node[circle, draw, fill=black!50,inner sep=0pt, minimum width=4pt] {};
 \draw[black] (0.5,-0.6) node[below] {$\vecz_0$};
  \draw[black] (0.5,0.5) node[circle, draw, fill=black!50,inner sep=0pt, minimum width=4pt] {};
 \draw[black] (0.6,0.6) node[left] {$\vecz_1$};
  \draw[black] (1,1) node[circle, draw, fill=black!50,inner sep=0pt, minimum width=4pt] {};
 \draw[black] (1,0.9) node[right] {$\vecz_2$};
 \draw[black] (0,2.5) node[circle, draw, fill=black!50,inner sep=0pt, minimum width=4pt] {};
 \draw[black] (0,2.5) node[left] {$\vecz_3$};
 \draw[black] (1,3) node[circle, draw, fill=black!50,inner sep=0pt, minimum width=4pt] {};
 \draw[black] (1,3) node[above] {$\vecz_4$};
 \draw[black] (4.25,2) node[circle, draw, fill=black!50,inner sep=0pt, minimum width=4pt] {};
 \draw[black] (4.4,2) node[above] {$\vecz_5$};
 \draw[black] (2,2) node[circle, draw, fill=black!50,inner sep=0pt, minimum width=4pt] {};
 \draw[black] (1.65,2) node[above] {$\vecz_6 = \vecy$};
 \draw[black] (1.4,1.4) node[circle, draw, fill=black!50,inner sep=0pt, minimum width=4pt] {};
 \draw[black] (1.45,1.5) node[left] {$\vecz_7$};
 \draw[black] (3.5,-0.5) node[circle, draw, fill=black!50,inner sep=0pt, minimum width=4pt] {};
 \draw[black] (3.65,-0.5) node[below] {$\vecz_8$};
 \draw[black] (0,1.3) node[circle, draw, fill=black!50,inner sep=0pt, minimum width=4pt] {};
 \draw[black] (0,1.3) node[left] {$\vecz_9$};
 \draw[black] (-1,0.25) node[circle, draw, fill=black!50,inner sep=0pt, minimum width=4pt] {};
 \draw[black] (-1.1,0.2) node[below] {$\vecz_{10}$};

 \draw[black] (0,0) node[circle, draw, fill=black!50,inner sep=0pt, minimum width=4pt] {};
 \draw[black] (0,-0.1) node[below] {$\vecx$};

 \draw[black] (-2,-1) node {$\tau$};

\end{tikzpicture}

\else
%
%

\begin{tikzpicture}[scale=0.75]

 \draw[thick,black] (0,0) -- (2,-2) -- (5,1) -- (3.5,3) -- (0,3) -- (0,0) -- (-2,0.5) -- (-1.5,-2) -- (0,0);
 \draw[thick,black] (3.5,3) -- (2,2);
 \draw[ultra thick]  (0,0) -- (2,2);

 \draw[black] (0,0) node[circle, draw, fill=black!50,inner sep=0pt, minimum width=4pt] {};
 \draw[black] (0,-0.1) node[below] {$\vecx$};
 \draw[black] (2,2) node[circle, draw, fill=black!50,inner sep=0pt, minimum width=4pt] {};
 \draw[black] (1.9,2) node[above] {$\vecy$};

 \draw[black] (-2,-1) node {$\tau$};
 \draw[black] (2.3,1.75) node[below] {$s = \ls[\vecx,\vecy]$};

\end{tikzpicture}
\hspace{2cm}
\begin{tikzpicture}[scale=0.75]

 \draw[thick,black] (0,0) -- (2,-2) -- (5,1) -- (3.5,3) -- (0,3) -- (0,0) -- (-2,0.5) -- (-1.5,-2) -- (0,0);
 \draw[thick,black] (3.5,3) -- (2,2);
 \draw[ultra thick]  (0,0) -- (2,2);
 \draw[dashed] (0.5,-0.5) -- (0.5,0.5) -- (1,1) -- (0,2.5) -- (1,3) -- (4.25,2) -- (2,2) -- (1.4,1.4) -- (3.5,-0.5) -- (0,1.3) -- (-1,0.25);

 \draw[black] (0.5,-0.5) node[circle, draw, fill=black!50,inner sep=0pt, minimum width=4pt] {};
 \draw[black] (0.5,-0.6) node[below] {$\vecz_0$};
  \draw[black] (0.5,0.5) node[circle, draw, fill=black!50,inner sep=0pt, minimum width=4pt] {};
 \draw[black] (0.6,0.6) node[left] {$\vecz_1$};
  \draw[black] (1,1) node[circle, draw, fill=black!50,inner sep=0pt, minimum width=4pt] {};
 \draw[black] (1,0.9) node[right] {$\vecz_2$};
 \draw[black] (0,2.5) node[circle, draw, fill=black!50,inner sep=0pt, minimum width=4pt] {};
 \draw[black] (0,2.5) node[left] {$\vecz_3$};
 \draw[black] (1,3) node[circle, draw, fill=black!50,inner sep=0pt, minimum width=4pt] {};
 \draw[black] (1,3) node[above] {$\vecz_4$};
 \draw[black] (4.25,2) node[circle, draw, fill=black!50,inner sep=0pt, minimum width=4pt] {};
 \draw[black] (4.4,2) node[above] {$\vecz_5$};
 \draw[black] (2,2) node[circle, draw, fill=black!50,inner sep=0pt, minimum width=4pt] {};
 \draw[black] (1.65,2) node[above] {$\vecz_6 = \vecy$};
 \draw[black] (1.4,1.4) node[circle, draw, fill=black!50,inner sep=0pt, minimum width=4pt] {};
 \draw[black] (1.45,1.5) node[left] {$\vecz_7$};
 \draw[black] (3.5,-0.5) node[circle, draw, fill=black!50,inner sep=0pt, minimum width=4pt] {};
 \draw[black] (3.65,-0.5) node[below] {$\vecz_8$};
 \draw[black] (0,1.3) node[circle, draw, fill=black!50,inner sep=0pt, minimum width=4pt] {};
 \draw[black] (0,1.3) node[left] {$\vecz_9$};
 \draw[black] (-1,0.25) node[circle, draw, fill=black!50,inner sep=0pt, minimum width=4pt] {};
 \draw[black] (-1.1,0.2) node[below] {$\vecz_{10}$};

 \draw[black] (0,0) node[circle, draw, fill=black!50,inner sep=0pt, minimum width=4pt] {};
 \draw[black] (0,-0.1) node[below] {$\vecx$};

 \draw[black] (-2,-1) node {$\tau$};

\end{tikzpicture}

\fi

\caption{An example to illustrate the steps in the proof of Theorem~\ref{LG-norm}. Here $S = [\vecz_0,\dots,\vecz_{10}]$ is a list of points in $\tau_1 = \tau \cup s$.
In this example, $I_1 = \{4,5,6,9,10\}$, $I_2 = \{2,7\}$ and $I_3 = \{1,3,8\}$. The sublists are $S_1 = [\vecz_3,\vecz_4,\vecx_5,\vecx_6,\vecz_8,\vecz_9,\vecx_{10}]$ and $S_2 = [\vecx_1,\vecx_2,\vecx_6,\vecx_7]$.}\label{ex1}

\end{center}
\end{figure}
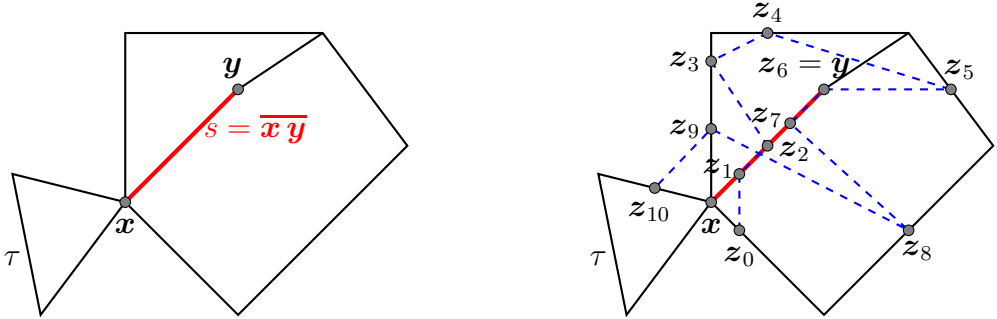

%
%

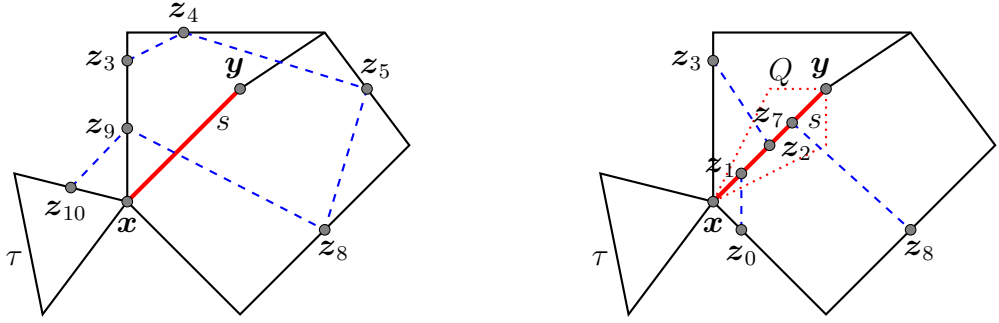
\begin{figure}
\begin{center}

%
%
\ifcolor

\begin{tikzpicture}[scale=0.75]

 \draw[thick,black] (0,0) -- (2,-2) -- (5,1) -- (3.5,3) -- (0,3) -- (0,0) -- (-2,0.5) -- (-1.5,-2) -- (0,0);
 \draw[thick,black] (3.5,3) -- (2,2);
 \draw[ultra thick,red]  (0,0) -- (2,2);

 \draw[thick,dashed,blue] (0,2.5) -- (1,3) -- (4.25,2) -- (3.5,-0.5) -- (0,1.3) -- (-1,0.25);

 \draw[black] (0,2.5) node[circle, draw, fill=black!50,inner sep=0pt, minimum width=4pt] {};
 \draw[black] (0,2.5) node[left] {$\vecz_3$};
 \draw[black] (1,3) node[circle, draw, fill=black!50,inner sep=0pt, minimum width=4pt] {};
 \draw[black] (1,3) node[above] {$\vecz_4$};
 \draw[black] (4.25,2) node[circle, draw, fill=black!50,inner sep=0pt, minimum width=4pt] {};
 \draw[black] (4.4,2) node[above] {$\vecz_5$};
 \draw[black] (3.5,-0.5) node[circle, draw, fill=black!50,inner sep=0pt, minimum width=4pt] {};
 \draw[black] (3.65,-0.5) node[below] {$\vecz_8$};
 \draw[black] (0,1.3) node[circle, draw, fill=black!50,inner sep=0pt, minimum width=4pt] {};
 \draw[black] (0,1.3) node[left] {$\vecz_9$};
 \draw[black] (-1,0.25) node[circle, draw, fill=black!50,inner sep=0pt, minimum width=4pt] {};
 \draw[black] (-1.1,0.2) node[below] {$\vecz_{10}$};

 \draw[black] (0,0) node[circle, draw, fill=black!50,inner sep=0pt, minimum width=4pt] {};
 \draw[black] (0,-0.1) node[below] {$\vecx$};
 \draw[black] (2,2) node[circle, draw, fill=black!50,inner sep=0pt, minimum width=4pt] {};
 \draw[black] (1.9,2) node[above] {$\vecy$};

 \draw[black] (-2,-1) node {$\tau$};
 \draw[black] (1.7,1.7) node[below] {$s$};

\end{tikzpicture}
\hspace{2cm}
\begin{tikzpicture}[scale=0.75]

 \draw[thick,black] (0,0) -- (2,-2) -- (5,1) -- (3.5,3) -- (0,3) -- (0,0) -- (-2,0.5) -- (-1.5,-2) -- (0,0);
 \draw[thick,black] (3.5,3) -- (2,2);
 \draw[ultra thick,red]  (0,0) -- (2,2);
 \draw[thick,dotted,red]  (0,0) -- (2,1) -- (2,2) -- (1,2) -- (0,0);
 \draw[thick,dashed,blue] (0.5,-0.5) -- (0.5,0.5);
 \draw[thick,dashed,blue] (1,1) -- (0,2.5);
 \draw[thick,dashed,blue] (1.4,1.4) -- (3.5,-0.5);

 \draw[black] (0.5,-0.5) node[circle, draw, fill=black!50,inner sep=0pt, minimum width=4pt] {};
 \draw[black] (0.5,-0.6) node[below] {$\vecz_0$};
  \draw[black] (0.5,0.5) node[circle, draw, fill=black!50,inner sep=0pt, minimum width=4pt] {};
 \draw[black] (0.6,0.6) node[left] {$\vecz_1$};
  \draw[black] (1,1) node[circle, draw, fill=black!50,inner sep=0pt, minimum width=4pt] {};
 \draw[black] (1,0.9) node[right] {$\vecz_2$};
 \draw[black] (0,2.5) node[circle, draw, fill=black!50,inner sep=0pt, minimum width=4pt] {};
 \draw[black] (0,2.5) node[left] {$\vecz_3$};
  \draw[black] (1.4,1.4) node[circle, draw, fill=black!50,inner sep=0pt, minimum width=4pt] {};
 \draw[black] (1.45,1.5) node[left] {$\vecz_7$};
 \draw[black] (3.5,-0.5) node[circle, draw, fill=black!50,inner sep=0pt, minimum width=4pt] {};
 \draw[black] (3.65,-0.5) node[below] {$\vecz_8$};

 \draw[black] (0,0) node[circle, draw, fill=black!50,inner sep=0pt, minimum width=4pt] {};
 \draw[black] (0,-0.1) node[below] {$\vecx$};
 \draw[black] (2,2) node[circle, draw, fill=black!50,inner sep=0pt, minimum width=4pt] {};
 \draw[black] (1.9,2) node[above] {$\vecy$};

 \draw[black] (-2,-1) node {$\tau$};
 \draw[black] (1.8,1.75) node[below] {$s$};
 \draw[black] (1.2,2.3) node {$Q$};
 \end{tikzpicture}

%
%
\else

\begin{tikzpicture}[scale=0.75]

 \draw[thick,black] (0,0) -- (2,-2) -- (5,1) -- (3.5,3) -- (0,3) -- (0,0) -- (-2,0.5) -- (-1.5,-2) -- (0,0);
 \draw[thick,black] (3.5,3) -- (2,2);
 \draw[ultra thick]  (0,0) -- (2,2);
 \draw[dashed] (0,2.5) -- (1,3) -- (4.25,2) -- (3.5,-0.5) -- (0,1.3) -- (-1,0.25);

 \draw[black] (0,2.5) node[circle, draw, fill=black!50,inner sep=0pt, minimum width=4pt] {};
 \draw[black] (0,2.5) node[left] {$\vecz_3$};
 \draw[black] (1,3) node[circle, draw, fill=black!50,inner sep=0pt, minimum width=4pt] {};
 \draw[black] (1,3) node[above] {$\vecz_4$};
 \draw[black] (4.25,2) node[circle, draw, fill=black!50,inner sep=0pt, minimum width=4pt] {};
 \draw[black] (4.4,2) node[above] {$\vecz_5$};
 \draw[black] (3.5,-0.5) node[circle, draw, fill=black!50,inner sep=0pt, minimum width=4pt] {};
 \draw[black] (3.65,-0.5) node[below] {$\vecz_8$};
 \draw[black] (0,1.3) node[circle, draw, fill=black!50,inner sep=0pt, minimum width=4pt] {};
 \draw[black] (0,1.3) node[left] {$\vecz_9$};
 \draw[black] (-1,0.25) node[circle, draw, fill=black!50,inner sep=0pt, minimum width=4pt] {};
 \draw[black] (-1.1,0.2) node[below] {$\vecz_{10}$};

 \draw[black] (0,0) node[circle, draw, fill=black!50,inner sep=0pt, minimum width=4pt] {};
 \draw[black] (0,-0.1) node[below] {$\vecx$};
 \draw[black] (2,2) node[circle, draw, fill=black!50,inner sep=0pt, minimum width=4pt] {};
 \draw[black] (1.9,2) node[above] {$\vecy$};
 
 \draw[black] (-2,-1) node {$\tau$};
 \draw[black] (1.7,1.7) node[below] {$s$};

\end{tikzpicture}
\hspace{2cm}
\begin{tikzpicture}[scale=0.75]

 \draw[thick,black] (0,0) -- (2,-2) -- (5,1) -- (3.5,3) -- (0,3) -- (0,0) -- (-2,0.5) -- (-1.5,-2) -- (0,0);
 \draw[thick,black] (3.5,3) -- (2,2);
 \draw[ultra thick]  (0,0) -- (2,2);
 \draw[dotted]  (0,0) -- (2,1) -- (2,2) -- (1,2) -- (0,0);
 \draw[dashed] (0.5,-0.5) -- (0.5,0.5);
 \draw[dashed] (1,1) -- (0,2.5);
 \draw[dashed] (1.4,1.4) -- (3.5,-0.5);

 \draw[black] (0.5,-0.5) node[circle, draw, fill=black!50,inner sep=0pt, minimum width=4pt] {};
 \draw[black] (0.5,-0.6) node[below] {$\vecz_0$};
  \draw[black] (0.5,0.5) node[circle, draw, fill=black!50,inner sep=0pt, minimum width=4pt] {};
 \draw[black] (0.6,0.6) node[left] {$\vecz_1$};
  \draw[black] (1,1) node[circle, draw, fill=black!50,inner sep=0pt, minimum width=4pt] {};
 \draw[black] (1,0.9) node[right] {$\vecz_2$};
 \draw[black] (0,2.5) node[circle, draw, fill=black!50,inner sep=0pt, minimum width=4pt] {};
 \draw[black] (0,2.5) node[left] {$\vecz_3$};
  \draw[black] (1.4,1.4) node[circle, draw, fill=black!50,inner sep=0pt, minimum width=4pt] {};
 \draw[black] (1.45,1.5) node[left] {$\vecz_7$};
 \draw[black] (3.5,-0.5) node[circle, draw, fill=black!50,inner sep=0pt, minimum width=4pt] {};
 \draw[black] (3.65,-0.5) node[below] {$\vecz_8$};

 \draw[black] (0,0) node[circle, draw, fill=black!50,inner sep=0pt, minimum width=4pt] {};
 \draw[black] (0,-0.1) node[below] {$\vecx$};
 \draw[black] (2,2) node[circle, draw, fill=black!50,inner sep=0pt, minimum width=4pt] {};
 \draw[black] (1.9,2) node[above] {$\vecy$};
 
 \draw[black] (-2,-1) node {$\tau$};
 \draw[black] (1.8,1.75) node[below] {$s$};
 \draw[black] (1.2,2.3) node {$Q$};
 \end{tikzpicture}
 
\fi

\caption{The segments for the sublist $S_1$ and the segments $\overline{\vecz_{i-1}\,\vecz_i}$ for $i \in I_3$.}\label{S1}
\end{center}
\end{figure}

Noting that $I_1 \cap I_2$ may be nonempty,
  \begin{equation}\label{eq1}
   \cvar(f,S)  \le \sum_{i = 1}^3 \sum_{j \in I_i} |f(\vecz_j) - f(\vecz_{j-1})|
  \end{equation}
(where empty sums are taken to be zero).
Form the sublist $S_1$ of $S$ by including all points which are endpoints of line segments $\ell_j$ with $j \in I_1$ (see Figure~\ref{S1}). By \cite[Proposition~3.5]{DL2} (see also \cite{AD4}), $\vf(S_1) \le \vf(S)$, and so
  \begin{equation}\label{eq2}
  \frac{\sum_{j \in I_1} |f(\vecz_j) - f(\vecz_{j-1})|}{\vf(S)}
   \le \frac{\cvar(f,S_1)}{\vf(S_1)}
   \le \var(f,\tau).
  \end{equation}
Similarly, if $S_2$ is the sublist of $S$ including all points which are endpoints of line segments $\ell_j$ with $j \in I_2$ then
  \begin{equation}\label{eq3}
  \frac{\sum_{j \in I_2} |f(\vecz_j) - f(\vecz_{j-1})|}{\vf(S)}
   \le \frac{\cvar(f,S_2)}{\vf(S_2)}
   \le \var(f,s).
  \end{equation}

Finally, one may draw a quadrilateral $Q$ with vertices at $\vecx$ and $\vecy$ which contains $s\setminus \{\vecx,\vecy\}$ in its interior and $\tau\setminus \{\vecx,\vecy\}$ in its exterior. If $j \in I_3$, then $\ell_j$ is a crossing segment on at least one of the four lines which determine $Q$. In particular, at least one of the four lines must have at least $\frac{1}{4} |I_3|$ crossing segments, and hence $\vf(S) \ge \frac{1}{4} |I_3|$. By a simple triangle inequality estimate
    \begin{equation}\label{eq4}
  \frac{\sum_{j \in I_3} |f(\vecz_j) - f(\vecz_{j-1})|}{\vf(S)}
   \le \frac{2 |I_3| \norm{f}_{\infty,\tau_1} }{\frac{1}{4} |I_3|} = 8 \norm{f}_{\infty,\tau_1}.
  \end{equation}
Combining the three estimates (which are of course trivially true if we are summing over an empty set of indices), and taking the supremum over all lists $S$, we see that
  \[\var(f,\tau_1) \le \var(f,\tau) + \var(f,s) + 8 \norm{f}_{\infty,\tau_1}\]
and hence (using the induction hypothesis)
  \begin{align*}
   \norm{f}_{BV(\tau_1)} &= \norm{f}_{\infty,\tau_1} +  \var(f,\tau_1) \\
  &\le \norm{f}_{\infty,\tau} + \norm{f}_{\infty,s} + \var(f,\tau) + \var(f,s) + 8   \norm{f}_{\infty,\tau_1} \\
  &\le C_m\Bigl(\norm{f}_{\infty,\tau} + \sum_{j=1}^m \var(f,s_j)\Bigr)
      + 9 \norm{f}_{\infty,\tau_1} + \var(f,s) \\
  &\le (C_m+9) \norm{f}_{LG(\tau_1)}.
   \end{align*}
This completes the induction proof.
 
\end{proof}

The proof above shows that if $\sigma = \bigcup_{j=1}^m s_j$ then $\norm{f}_{BV(\sigma)} \le 9m \norm{f}_{LG(\sigma)}$. While the factor 9 is unlikely to be sharp, one cannot in general do better than a linear dependence on $m$.

We shall need the following simple result in Section~\ref{AC-isom}.

\begin{corollary}\label{LG-decomp}
Suppose that $\sigma = \bigcup_{j=1}^m s_j \in \LG$, and that $\sigma_1 = \bigcup_{j=1}^k s_j$ and $\sigma_2 = \bigcup_{j=k+1}^m s_j$ are also in $\LG$. Then for all $f \in BV(\sigma)$
  \[ \max(\norm{f}_{LG(\sigma_1)},\norm{f}_{LG(\sigma_2)} )
  \le \norm{f}_{LG(\sigma)}
  \le \norm{f}_{LG(\sigma_1)} + \norm{f}_{LG(\sigma_2)}.  \]
\end{corollary}


 \section{Graphs and graph drawings}

Our proof of Theorem~\ref{main-theorem} uses a number of ideas from graph theory. In doing so we shall need to switch between our analyst's view of a set $\sigma$ sitting as a subset of the plane, and a more graph theoretic viewpoint, where a graph $G$ consists of a (possibly abstract) collection of vertices and points.  Roughly speaking, a set $\sigma \subseteq \C$ comes with topological information inherited as a subset of the plane, but without graph theoretic properties such as vertices and edges. On the other hand, a general graph may lack the topological structure suitable for our purposes. By restricting the class of objects to linear graphs, we can apply ideas from both areas. The challenge is to match up the different senses of `isomorphism'.

We shall generally stick with the standard graph theoretic notation and terminology, as may be found in references such as \cite{Die}. To simplify matters, we shall restrict ourselves to the setting of planar graphs. Thus an undirected, simple graph $G = (V,E)$ consists of a set $V$ of vertices, and a set $E$ comprising two element subsets of $V$ which are the edges of $G$. We shall only consider finite, undirected, simple graphs here, so the reader should assume these properties unless otherwise indicated.

A drawing of a planar graph $G = (V,E)$ consists of a set of points ${\hat V} = \{\vecx_i\}_{i=1}^n$ in $\R^2$ and a set of smooth curves ${\hat E} = \{\gamma_j\}_{j=1}^m$ in $\R^2$ such that
\begin{itemize}
 \item for $1 \le i < j \le n$, there is a curve joining $\vecx_i$ and $\vecx_j$ if and only if there is an edge joining $v_i$ and $v_j$,
 \item distinct curves in ${\hat E}$ do not intersect, except possibly at the endpoints.
\end{itemize}
We say that  ${\hat V}$ and ${\hat E}$ represent the vertices and edges of $G$.
The distinction between a graph and a drawing of the graph is of course often blurred, as is the distinction between the drawing and the subset of the plane formed by the union of the curves ${\hat E}$.

For our application in proving Theorem~\ref{main-theorem}, we start with an element $\sigma \in \LG$ and need to impose some graph theoretic structure on the set. Suppose that $\sigma \in \LG$ has proper representation  $\sigma = \bigcup_{j=1}^m \ls[\vecx_j,\vecx_j']$. The set $\sigma$ can then be thought of as providing a drawing of a graph whose vertex set consists of the endpoints of the line segments $\{\vecx_j,\vecx_j'\}_{j=1}^m$ and whose edge set is determined by the line segments. Of course, different representations will produce different graphs, but if the representation is understood, we shall use the notation $G(\sigma) = (V(\sigma),E(\sigma))$ to denote the graph generated by the representation.

Our next step is to match the topological notion of homeomorphism between subsets of the plane and the appropriate graph theoretic notions.

\subsection{Graph isomorphisms}

Recall that
 two simple graphs $ G_{1} = (V_{1},E_{1}) $ and $ G_{2} = (V_{2},E_{2}) $ are called (graph) isomorphic if there exists a bijective mapping, $ h : V_{1} \to V_{2} $ such that there is a edge between $u$ and $v$ in $G_{1}$, if and only if there exists an edge between $ h(u)$ and $h(v)$ in $G_{2} $.

\begin{definition}
A \textit{subdivision} of an edge $\{u,v\}$ of a graph $G$ comprises forming a new graph with an additional vertex $w$, and replacing the edge $\{u,v\}$ with the two edges $\{u,w\}$ and $\{w,v\}$. A \textit{subdivision} of $G$ is a graph formed by starting with $G$ and performing a finite sequence of subdivisions of edges.
\end{definition}

\begin{definition}
Two graphs  $ G_{1} $ and  $ G_{2} $ are \textit{(graph) homeomorphic} if there is a graph isomorphism from some subdivision of $ G_{1} $ to some subdivision of $ G_{2} $.
\end{definition}

For example, the graphs
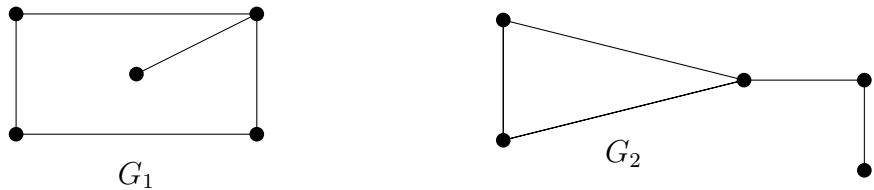
\begin{center}
\begin{tikzpicture}[scale=0.8]%

\draw (-2,-2)--(-2,0)--(2,0)--(2,-2)--(-2,-2); 
\draw (2,0) -- (0,-1);
\draw(-2,0) node[circle,fill,black,scale=0.5,label=above:{}]
{};
\draw (2,0) node[circle,fill,black,scale=0.5,label=above:{}]{};
\draw (2,-2) node[circle,fill,black,scale=0.5,label=below:{}]
{};
\draw (-2,-2) node[circle,fill,black,scale=0.5,label=left:{}]{};
\draw (0,-1) node[circle,fill,black,scale=0.5,label=below:{}]{};
\draw (0,-2.1) node[label=below:{$G_1$}]{};

\end{tikzpicture}%
\hspace{3cm}%
\begin{tikzpicture}[scale=0.8]%
\draw(-2,0)--(-2,-2)--(-2,0)--(2,-1)--(-2,-2)--(2,-1)--(4,-1)--(4,-2.5);
\draw(-2,0) node[circle,fill,black,scale=0.5,label=above:{}]
{};
\draw (-2,-2) node[circle,fill,black,scale=0.5,label=above:{}]{};
\draw (2,-1) node[circle,fill,black,scale=0.5,label=below:{}]{};

\draw(4,-1) node[circle,fill,black,scale=0.5,label=below:{}]{};

\draw(4,-2.5) node[circle,fill,black,scale=0.5,label=above:{}]{};

\draw(0,-2.8) node[label=above:{$G_2$}]{};

\end{tikzpicture}%
\end{center}
are graph homeomorphic via the following subdivisions and the graph isomorphism which sends each vertex $v$ to the corresponding vertex $\hat v$.
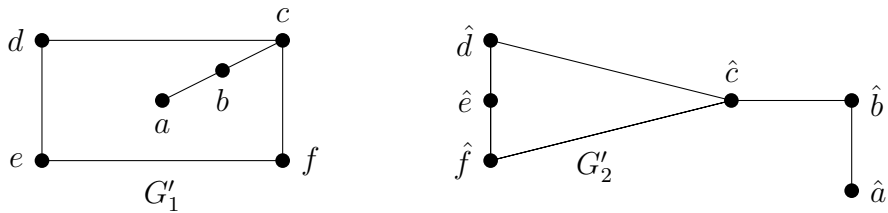
 \begin{center}
 \begin{tikzpicture}[scale=0.8]%
 \draw (-2,-2)--(-2,0)--(2,0)--(2,-2)--(-2,-2); 
 \draw (2,0) -- (0,-1);
 \draw(-2,0) node[circle,fill,black,scale=0.5,label=left:{$ d $}] {};
 \draw (2,0) node[circle,fill,black,scale=0.5,label=above:{$ c $}]{};
 \draw (2,-2) node[circle,fill,black,scale=0.5,label=right:{$ f $}] {};
 \draw (-2,-2) node[circle,fill,black,scale=0.5,label=left:{$ e $}]{};
 \draw (0,-1) node[circle,fill,black,scale=0.5,label=below:{$ a $}]{};
 \draw (1,-0.5) node[circle,fill,black,scale=0.5,label=below:{$ b $}]{};
  \draw (0,-3.2) node[label=above:{$G_1' $}]{};

 \end{tikzpicture}%
 \hspace{1.5cm}%
 \begin{tikzpicture}[scale=0.8]%
\draw (-2,0)--(-2,-2)--(-2,0)--(2,-1)--(-2,-2)--(2,-1)--(4,-1)--(4,-2.5);
\draw(-2,0) node[circle,fill,black,scale=0.5,label=left:{$ \hat{d} $}]{};
\draw (-2,-2) node[circle,fill,black,scale=0.5,label=left:{$ \hat{f} $}]{};
\draw (2,-1) node[circle,fill,black,scale=0.5,label=above:{$ \hat{c} $}]{};

\draw(4,-1) node[circle,fill,black,scale=0.5,label=right:{$ \hat{b} $}]{};

\draw(4,-2.5) node[circle,fill,black,scale=0.5,label=right:{$ \hat{a} $}]{};
\draw(-2,-1) node[circle,fill,black,scale=0.5,label=left:{$ \hat{e} $}]{};
\draw(0.4,-2.1) node[label=left:{$G_2' $}]{};
\draw (0,-3.2) node {};
 \end{tikzpicture}%
 \end{center}

Fortunately, at least amongst the graphs which concern us, the two notions of homeomorphism agree.

 \begin{theorem}[{\cite[page 18]{GT}}]
Suppose that $G_1$ and $G_2$ are planar graphs with drawings ${\hat G}_1$ and ${\hat G}_2$ in the plane. Then
$G_1$ and $G_2$ are graph homeomorphic if and only if ${\hat G}_1$ and ${\hat G}_2$ are topologically homeomorphic.
 \end{theorem}

In fact we shall only use one direction of this result. For completeness, we include a proof of this implication.

\begin{corollary}\label{homeo}
Suppose that $\sigma,\tau \in \LG$ have proper representations which generate graphs $G(\sigma) = (V(\sigma),E(\sigma))$ and $G(\tau) = (V(\tau),E(\tau))$.
If $\sigma$ and $\tau$ are topologically homeomorphic then $G(\sigma)$ and $G(\tau)$ are graph homeomorphic.
\end{corollary}

\begin{proof}
Suppose that $h: \sigma \to \tau$ is a homeomorphism.
Define sets
  \[
   {\hat V}(\sigma) = V(\sigma) \cup h^{-1}(V(\tau)), \qquad \qquad
   {\hat V}(\tau) = h(V(\sigma)) \cup V(\tau).
  \]
It is clear that $h$ maps ${\hat V}(\sigma)$ bijectively to ${\hat V}(\tau)$.

If $\vecu,\vecv \in {\hat V}(\sigma)$, then we shall say that $\{\vecu,\vecv \} \in {\hat E}(\sigma)$ if there exists a continuous curve $\gamma: [0,1] \to \sigma$ such that $\gamma(0) = \vecu$, $\gamma(1) = \vecv$ and $\gamma(t) \not\in {\hat V}(\sigma)$ for $0 < t < 1$. The edge set ${\hat E}(\tau)$ is defined similarly, and so we can consider the two graphs ${\hat G}(\sigma) = \bigl({\hat V}(\sigma),{\hat E}(\sigma)\bigr)$ and ${\hat G}(\tau) = \bigl({\hat V}(\tau),{\hat E}(\tau)\bigr)$. Clearly $\gamma$ is a curve joining $\vecu,\vecv \in {\hat V}(\sigma)$ if and only if $h \circ \gamma$ is a curve joining $h(\vecu),h(\vecv) \in {\hat V}(\tau)$. It follows that the graphs ${\hat G}(\sigma)$ and ${\hat G}(\tau)$ are isomorphic.  Since ${\hat G}(\sigma)$ is a subdivision of $G(\sigma)$ and ${\hat G}(\tau)$ is a subdivision of $G(\tau)$, the graphs  $G(\sigma)$ and $G(\tau)$ are graph homeomorphic.
 
\end{proof}

We note that there is a distinct graph theoretic notion of topological isomorphism (see \cite{GT}).
Homeomorphic planar graphs need not be topologically isomorphic.

We can now show that the $LG(\sigma)$ norm is independent of the representation of $\sigma$ as a union of line segments.

\begin{theorem}\label{indep-rep}
Suppose that $\sigma = \bigcup_{j=1}^m s_j = \bigcup_{k=1}^n t_k$ are two proper representations of $\sigma \in \LG$. If $f: \sigma \to \C$ then
   $\displaystyle  \sum_{j=1}^m \var(f,s_j) = \sum_{k=1}^n \var(f,t_k)$.
\end{theorem}

\begin{proof}
Let $G_1$ be the graph generated by the representation $\bigcup_{j=1}^m s_j$. If one takes a subdivision of (the edge represented by) $s_j = \ls[\vecx,\vecy]$ then one replaces $s_j$ with two smaller line segments $s_{j,1} = \ls[\vecx,\vecw]$ and $s_{j,2} = \ls[\vecw,\vecy]$ which produces another proper representation. Furthermore, the usual properties of variation on an interval ensure that $\var(f,s_j) = \var(f,s_{j,1}) + \var(f,s_{j,2})$.

Let $G_2$ is the graph generated by the representation $\bigcup_{k=1}^n t_k$.  Following the proof of Corollary~\ref{homeo} (with the identity homeomorphism) shows that there is a graph $G$ which is a subdivision of both $G_1$ and $G_2$, and which determines a proper representation $\sigma = \bigcup_{i=1}^r {\hat s}_i$. By the earlier comment, $\sum_{j=1}^m \var(f,s_j) = \sum_{i=1}^r \var(f,{\hat s}_i) = \sum_{k=1}^n \var(f,t_k)$.
 
\end{proof}

 \subsection{Subgraphs}

The proof of Theorem~\ref{main-theorem} proceeds via an inductive procedure involving subgraphs of the original graphs.

\begin{definition}
 A graph $ \hat{G} $ is a \textit{subgraph} of a graph $ G $ if each of its vertices belongs to $ V(G) $ and each of its edges belongs to $ E(G) $.
 \end{definition}

Standard algorithms for finding minimal spanning trees show that if $G$ is a connected graph with $n$ edges then there exist subgraphs $G_{1}\subseteq  G_{2}\subseteq\dots\subseteq G_{n}=G$, when each $G_{k}$ is a connected graph with $k$ edges. We shall call such a sequence an
\textit{edge-by-edge decomposition} of $ G $. The following result is straightforward.

 \begin{lemma}\label{edges}
 Suppose that $G$ and $\hat{G}$ are isomorphic graphs with $ n$ edges, with  graph isomorphism $ f$. If $\{G_{k}\}_{k=1}^{n}$ is an edge-by-edge decomposition of $ G $ then $\{f(G_{k})\}_{k=1}^{n}$ is an edge-by-edge decomposition of $\hat{G}$.
  \end{lemma}

\section{Algebra homomorphisms for $BV$ functions}

The results of the last two sections allow an easy proof of the following theorem.

\begin{theorem}\label{BV-isom}
Suppose that $\sigma,\tau \in \LG$ are homeomorphic. Then $BV(\sigma)$ is isomorphic to $BV(\tau)$ (as Banach algebras).
\end{theorem}

\begin{proof}
By Corollary~\ref{homeo} (and the earlier comments) we may consider $\sigma$ and $\tau$ to be graph drawings of isomorphic graphs with each edge represented by a line segment.
By Lemma~\ref{edges} then, there are line segments $\{s_k\}_{k=1}^m$ and $\{t_k\}_{k=1}^m$ so that if $\sigma_j = \bigcup_{k=1}^j s_k$ and $\tau_j = \bigcup_{k=1}^j t_k$ then
$\{\sigma_j\}_{j=1}^m$ and $\{\tau_j\}_{j=1}^m$ are edge-by-edge decompositions of $\sigma$ and $\tau$, and such that each $\sigma_j$ is homeomorphic and graph isomorphic to $\tau_j$.

There are two affine homeomorphisms from $s_1$ to $t_1$. Fix one of these and call it $\phi_1$.
There is now a unique homeomorphism $\phi_2$ from $\sigma_2$ to $\tau_2$, extending $\phi_1$ which is also affine when restricted to $s_2$. Continuing in this way, one can construct a homeomorphism $\phi: \sigma \to \tau$ which is affine on each of the component line segments.

For $f: \sigma \to \C$, define $\Phi(f) : \tau \to \C$ by $\Phi(f)(\vecz) = f(\phi^{-1}(\vecz))$. Clearly $\Phi$ is linear and multiplicative. Since the sets are just line segments, for each $k$,
  \[ \var(f,s_k) = \var(\Phi(f),t_k). \]
Consequently, $\Phi$ maps $\bigl(BV(\sigma),\norm{\cdot}_{LG(\sigma)}\bigr)$ isometrically to $\bigl(BV(\tau),\norm{\cdot}_{LG(\tau)}\bigr)$.
By Theorem~\ref{LG-norm}, this implies that the spaces $BV(\sigma)$ and $BV(\tau)$ are isomorphic as Banach algebras.
 
\end{proof}

\section{$AC$ functions}\label{AC-isom}

It is not too hard to construct a Banach algebra isomorphism $\Phi: BV[0,1] \to BV[0,1]$ which does not restrict to an isomorphism from $AC[0,1] \to AC[0,1]$. Our aim in this section is to show that the map defined in the proof of Theorem~\ref{BV-isom} does preserve absolutely continuous functions.

In the proof we will need an absolutely continuous function which is almost the indicator function of one of the component line segments of $\sigma \in \LG$. For $0 < \delta < \frac{1}{2}$ let $k_\delta \in AC[0,1]$ be the function which is $1$ on $[\delta,1-\delta]$, $0$ on $[0,\delta/2] \cup [1-\delta/2,1]$ and which linearly interpolates the remaining values.

\begin{lemma}\label{pw-linear-h}
Suppose that $\sigma = \bigcup_{k=1}^m s_k \in \LG$ with $s_1 = [0,1]$. Let $0 < \delta < \frac{1}{2}$ and define $h: \sigma \to \R$ by
  \[ h(\vecz) = \begin{cases}
      0,      & \vecz \not\in s_1, \\
      k_\delta(\vecz),   & \vecz \in s_1.
               \end{cases}                                                                                                                                                                                                                                                                                                                              \]
Then $h \in AC(\sigma)$.
\end{lemma}

\begin{proof} Suppose that $\alpha >0$. By \cite[Proposition~4.4]{AD1}, the functions $h_1(x+iy) = k_\delta(x)$ and $h_2(x+iy) = \max(0,1-\alpha|y|)$ are in $AC(\sigma)$ and hence $h_1 h_2 \in AC(\sigma)$. But if $\alpha$ is sufficiently large, then $h_1 h_2 = k_\delta$ and so we are done.
 
\end{proof}

\begin{theorem}\label{AC-add-edge}
Suppose that $\sigma \in \LG$ and that $s$ is a line segment which intersects $\sigma$ at one or both of its endpoints (and nowhere else). Let
 $\sigma_1 = \sigma \cup s \in \LG$. Suppose that $f \in BV(\sigma_1)$. Then $f \in AC(\sigma_1)$ if and only if
\begin{enumerate}
 \item $f|\sigma \in AC(\sigma)$, and
 \item $f|s \in AC(s)$.
\end{enumerate}
\end{theorem}

\begin{proof}
Restricting an absolutely continuous function to a smaller compact set preserves absolute continuity (\cite[Lemma~4.5]{AD1}) so it is clear that if  $f \in AC(\sigma_1)$ then (1) and (2)  hold.

Suppose conversely that $f$ satisfies (1) and (2).  Recalling that affine transformations preserve absolute continuity, we may apply an affine transformation to the plane to reduce to the case that $s = [0,1]$.  Define $f_s: \sigma_1 \to \C$ by
  \[ f_s(\vecz) = \begin{cases}
                   f(0),        & \text{if $\Re \vecz < 0$}, \\
                   f(\Re \vecz), & \text{if $0 \le \Re \vecz \le 1$}, \\
                   f(1),        & \text{if $\Re \vecz > 1$}. \\
                  \end{cases} \]
By \cite[Proposition~4.4]{AD1}, $f_s \in AC(\sigma_1)$ since it is the extension to a subset of the plane of a function which is absolutely continuous on an  interval in $\R$. Let $g = f - f_s \in BV(\sigma_1)$. Note that (since $f_s|\sigma \in AC(\sigma)$)
\begin{itemize}
 \item[$\bullet$] $g|\sigma \in AC(\sigma)$, and
 \item[$\bullet$] $g|s = 0$.
\end{itemize}
Fix $\epsilon > 0$. Using Theorem~\ref{LG-norm}, it suffices now to show that there exists $q \in AC(\sigma_1)$ with $\norm{g-q}_{LG(\sigma_1)} < \epsilon$. This will imply that $g \in AC(\sigma_1)$ and hence that $f \in AC(\sigma_1)$ too.

By definition there exists a polynomial $p \in {\mathcal{P}}_2$ such that $\norm{g - p}_{LG(\sigma)} \le \frac{\epsilon}{5}$.
Suppose first that both $0$ and $1$ are elements of $\sigma$. Then
  \[ |p(0)| = |p(0) - g(0)| \le \norm{p-g}_{\infty} \le \norm{p-g}_{LG(\sigma)} < \frac{\epsilon}{5}, \]
with a similar bound on $|p(1)|$.
Since $p$ is absolutely continuous on $[0,1]$, there exists $\delta > 0$ such that $\var(p,[0,\delta]) < \frac{\epsilon}{5}$ and $\var(p,[1-\delta,1]) < \frac{\epsilon}{5}$. It follows that $|p(\vecz)| < \frac{2\epsilon}{5}$ for all $\vecz \in [0,\delta] \cup [1-\delta,1]$. Use Lemma~\ref{pw-linear-h} to define $h \in AC(\sigma_1)$ which is supported on $[\delta/2,1-\delta/2]$, and let $q = p(1-h)$. Then certainly $q \in AC(\sigma_1)$. Furthermore, noting that $q = p$ on $\sigma$, by Corollary~\ref{LG-decomp}
  \[ \norm{g-q}_{LG(\sigma_1)} \le \norm{g-q}_{LG(\sigma)} + \norm{g-q}_{LG(s)}
   = \norm{g-p}_{LG(\sigma)} + \norm{g-q}_{BV(s)}. \]
Now, using \cite[Proposition 3.7]{AD1}
  \begin{align*}
  \norm{g-q}_{BV(s)} &= \norm{q}_{BV(s)} \\
      &= \norm{q}_{\infty,s} + \var(p(1-h),[0,1]) \\
      &< \frac{2\epsilon}{5} + \var(p(1-h),[0,\delta]) \\
      & \qquad + \var(p(1-h),[\delta,1-\delta]) + \var(p(1-h),[1-\delta,1])\\
      &\le \frac{2\epsilon}{5} + \norm{1-h}_\infty \var(p,[0,\delta]) + 0 + \norm{1-h}_\infty \var(p,[1-\delta,1]) \\
      &< \frac{4\epsilon}{5}.
  \end{align*}
Thus $\norm{g-q}_{LG(\sigma_1)} < \frac{\epsilon}{5}+\frac{4\epsilon}{5} = \epsilon$ .

The case where $\sigma$ and $s$ meet at just one point, say $0$, is essentially the same, except that in this case one uses a function $h$ which is zero on all of $[\delta,1]$ rather than just on $[\delta,1-\delta]$.
 
\end{proof}

We can now complete the proof of Theorem~\ref{main-theorem}.

\begin{proof}
The fact that if $AC(\sigma)$ is isomorphic to $AC(\tau)$ then $\sigma$ is homeomorphic to $\tau$ is just Theorem~2.6 of \cite{DL1}.

Suppose conversely that $\sigma$ is homeomorphic to $\tau$. As in the proof of Theorem~\ref{BV-isom}, fix corresponding edge-by-edge decompositions $\{\sigma_j\}_{j=1}^m$ and $\{\tau_j\}_{j=1}^m$ of $\sigma$ and $\tau$, and let $\phi$ be a homeomorphism from $\sigma$ to $\tau$ which is affine on each of the component line segments.
Let $\Phi: BV(\sigma) \to BV(\tau)$ be the corresponding Banach algebra isomorphism $\Phi(f)(\vecz) = f(\phi^{-1}(\vecz))$.

Suppose that $f \in AC(\sigma)$. Then $f|\sigma_1 \in AC(\sigma_1)$ and since $\phi^{-1}$ is an affine map on $\tau_1$, it is clear that $f \circ \phi^{-1}$ is absolutely continuous on $\tau_1$. Repeated use of Theorem~\ref{AC-add-edge} now allows one to deduce that $f \circ \phi^{-1}|\sigma_j \in AC(\tau_j)$ for each $j$, and in particular, that $\Phi(f) \in AC(\tau)$.

Since $\Phi^{-1}(g)(\vecz) = g(\phi(z))$, the same proof shows that the image of $AC(\sigma)$ under $\Phi$ is all of $AC(\tau)$ and hence that $AC(\sigma)$ is isomorphic to $AC(\tau)$.
 
\end{proof}

Theorem~\ref{main-theorem} can be extended easily to the class of possibly disconnected unions of closed line segments in the plane. Let $L_m$ denote the collection of sets which can be written as a union of $m$ such line segments. Each such collection has finitely many  equivalence classes of homeomorphic sets. In particular this means that, up to isomorphism, there are only countably many spaces $AC(\sigma)$ with $\sigma \in \LG$.

\section{A final example}\label{fin-ex}

One might naturally ask the extent to which the one-dimensional structure of sets in $\LG$ is vital in the results above. Preliminary work indicates that the linear structure can be relaxed to some extent. The following example gives an indication of some of the restrictions which will be required in generalizing these results.

\begin{example}
Let $\sigma = [0,1] = \{(t,0) \,:\, 0 \le t \le 1 \}$ and let $\tau = \{(0,0) \} \cup \bigl\{\bigl(t,t \sin\frac{1}{t} \bigr) \,:\, 0 < t \le 1 \bigr\}$. Then $\sigma$ and $\tau$ are homeomorphic curves. Suppose that $AC(\sigma)$ is isomorphic to $AC(\tau)$ via an isomorphism $\Phi$. Then there exists a homeomorphism $h: \sigma \to \tau$ such that $\Phi(f)(\vecx) = f(h^{-1}(\vecx))$.

The function $g: \tau \to \C$,  $g(x,y) = y$ is in $\mathbb{P}_2(\tau)$ and hence $g \in AC(\tau)$. Let $f = \Phi^{-1}(g) = f \circ h$.
For $j=1,2,3,\dots$ let $t_j = 2/((2j-1)\pi)$, and let $\vecx_j = \bigl(t_j,t_j \sin \frac{1}{t_j} \bigr) \in \tau$. Let $\vecz_j = h^{-1}(\vecx_j)$ so that $f(\vecz_j) = (-1)^j t_j$. Since $h$ is a homeomorphism it is clear that the sequence $\{\vecz_j\}$ must be monotone in $[0,1]$. Since $\sigma$ is a line segment in $\R$ we can use the classical definition of variation and so for all $n$
  \[
  \var(f,\sigma) \ge \sum_{j=2}^n | f(\vecz_j) - f(\vecz_{j-1})|
     = \frac{2}{\pi} \sum_{j=2}^n \Bigl| \frac {1}{2j-1} + \frac{1}{2j-3} \Bigr|
  \]
and hence $f$ is not even of bounded variation on $\sigma$.

It follows that no such isomorphism $\Phi$ can exist. That is, $\sigma$ and $\tau$ are homeomorphic curves for which the function spaces $AC(\sigma)$ and $AC(\tau)$ are not isomorphic.
\end{example}

\medskip

\textbf{Acknowledgements}:
The work of the first author was financially supported
by the Ministry of Higher Education and Scientific Research
of Iraq.


\begin{thebibliography}{[10]}

\bibitem{AD1} B. Ashton and I. Doust,
     Functions of bounded variation on compact subsets
         of the plane,
     Studia Math.
     \textbf{169} (2005), 163--188.

\bibitem{AD4} B. Ashton and I. Doust,
      A comparison of algebras of functions of bounded variation,
        Proc. Edinb. Math. Soc. (2)
        \textbf{49} (2006), 575--591.

\bibitem{AD3} B. Ashton and I. Doust,
     Compact $AC(\sigma)$ operators,
     Integral Equations Operator Theory \textbf{63} (2009), 459-472.

\bibitem{AD2} B. Ashton and I. Doust,
    $AC(\sigma)$ operators,
    J. Operator Theory \textbf{65} (2011), 255--279.

 \bibitem{BG} E. Berkson and T. A. Gillespie,
 $AC$ functions on the circle and spectral families,
  J. Operator Theory \textbf{13} (1985), 33--47.

 \bibitem{Die} R. Diestel,
 Graph theory,
 Fourth edition, Graduate Texts in Mathematics  173,
 Springer, Heidelberg, 2010.

 \bibitem{DAS} I. Doust and S. Al-shakarchi,
 Isomorphisms of $AC(\sigma)$ spaces for countable sets,
  The diversity and beauty of applied operator theory, 193--206,
  Oper. Theory Adv. Appl., \textbf{268}, Birkh{\"a}user/Springer, Cham, 2018.


 \bibitem{DL1}  I. Doust and M. Leinert,
     Isomorphisms of $AC(\sigma)$ spaces,
  Studia Math. \textbf{228} (2015), 7--31.

 \bibitem{DL2}  I. Doust and M. Leinert,
    Approximation in $AC(\sigma)$, arXiv:1312.1806v1, 2013.

 \bibitem{Dow} H. R. Dowson,
  Spectral theory of linear operators,
  London Mathematical Society Monographs 12, Academic Press, London, 1978.

 \bibitem{DS} H. R. Dowson and P. G. Spain,
 An example in the theory of well-bounded operators,
 Proc. Amer. Math. Soc. \textbf{32} (1972), 205--208.

  \bibitem{GT} J. L. Gross and T. W. Tucker,
   Topological graph theory,
   Wiley, New York, 1987.

\end{thebibliography}
\end{document}